\documentclass[12pt]{amsart}
\usepackage{epsfig,color}

\usepackage{mathrsfs}

\theoremstyle{definition}

\newtheorem{theorem}{Theorem}[section]

\newtheorem{proposition}[theorem]{Proposition}
\newtheorem{lemma}[theorem]{Lemma}
\newtheorem{remark}[theorem]{Remark}
\newtheorem{corollary}[theorem]{Corollary}

\newtheorem*{acknowledgements}{Acknowledgements}

\numberwithin{equation}{section}

\DeclareMathOperator*{\supp}{spt}
\renewcommand\div{\operatorname{div}}

\newcommand{\ol}{\overline}
\newcommand{\Lap}{\Delta}
\newcommand{\ssubset}{\subset\subset}

\newcommand{\reg}{\operatorname{reg}}
\newcommand{\sing}{\operatorname{sing}}
\newcommand{\Ric}{\operatorname{Ric}}

\renewcommand*\d{\mathop{}\!\mathrm{d}}

\DeclareMathOperator*{\vol}{Vol}

\title{First stability eigenvalue of singular minimal hypersurfaces in spheres}

\author{Jonathan J. Zhu}%
\address{Department of Mathematics,
Harvard University, Cambridge, MA 02138, USA}
\email{jjzhu@math.harvard.edu}

\begin{document}
\date{\today}
\maketitle

\begin{abstract}

In this short note we extend an estimate due to J. Simons on the first stability eigenvalue of minimal hypersurfaces in spheres to the singular setting. Specifically, we show that any singular minimal hypersurface in $\mathbf{S}^{n+1}$, which is not totally geodesic and satisfies the $\alpha$-structural hypothesis, has first stability eigenvalue at most $-2n$, with equality if and only if it is a product of two round spheres. The equality case was settled independently in the classical setting by Wu and Perdomo. 
\end{abstract}

\section{Introduction}
\label{sec:introduction}

A hypersurface $M^n$ in a Riemannian manifold $N^{n+1}$ is said to be minimal if the first variation of area is zero, or equivalently if its mean curvature vanishes identically. The study of minimal hypersurfaces is one of the foremost topics in differential geometry. Minimal hypersurfaces in the round sphere $\mathbf{S}^{n+1}$ are of particular interest since they correspond to minimal cones in Euclidean space, which arise as blowup models for singularities of minimal hypersurfaces in general ambient spaces. 

A classical problem concerns the classification of \textit{stable} minimal cones in Euclidean space. An orientable minimal hypersurface $M^n \subset N^{n+1}$ is stable if the second variation of area is nonnegative, or equivalently if the first eigenvalue $\lambda_1=\lambda_1(L)$ of the Jacobi operator \begin{equation} L = \Lap_M + |A|^2 + \Ric^N(\nu,\nu)\end{equation} is nonnegative, where $A$ is the second fundamental form and $\nu$ is the unit normal on $M$.  

 In \cite{simonsjim}, J. Simons showed that the stability of a regular minimal cone $\Sigma^{n+1}=C(M)$ as a hypersurface in $\mathbf{R}^{n+2}$ can be reduced to the condition $\lambda_1(M) \geq -\frac{(n+1)^2}{4}$ on the link $M^n$, considered as a minimal hypersurface in $\mathbf{S}^{n+1}$. He furthermore showed that any smooth closed minimal hypersurface $M^n\subset\mathbf{S}^{n+1}$ that is not an equator must satisfy $\lambda_1 (M) \leq -2n$, which implies that there are no non-flat stable regular minimal cones in $\mathbf{R}^k$, for $k\leq 7$. 

In this short note we extend Simons' estimate to singular minimal hypersurfaces in $\mathbf{S}^{n+1}$: 

\begin{theorem}
\label{thm:stabilityeigsphintro}
Let $V$ be a stationary integral $n$-varifold in $\mathbf{S}^{n+1}$ with orientable regular part, and which satisfies the $\alpha$-structural hypothesis for some $\alpha\in(0,\frac{1}{2})$.

Suppose that $V$ is not totally geodesic in $\mathbf{S}^{n+1}$. Then $\lambda_1(V)\leq -2n$, with equality if and only if $V$ is an integer multiple of a Clifford hypersurface $\mathbf{S}^k\left(\sqrt{\frac{k}{n}}\right)\times \mathbf{S}^l\left(\sqrt{\frac{l}{n}}\right)$, $k+l=n$. 
\end{theorem}

The reader is directed to Section \ref{sec:notation} for the precise definitions. Here $\mathbf{S}^k(r)$ denotes a round sphere of radius $r$. The $\alpha$-structural hypothesis of Wickramasekera \cite{wick} is a technical condition that will allow us to use Wickramasekera's regularity theory to establish sufficient control on the singular set; it is automatically satisfied, for instance, if the singular set initially has vanishing codimension 1 Hausdorff measure. One should also note that the equality case in Theorem \ref{thm:stabilityeigsphintro} was characterised, in the classical setting, independently by Wu \cite{wu} and Perdomo \cite{perdomo}. 

The first key ingredient for Theorem \ref{thm:stabilityeigsphintro} is the famous Simons' inequality, which for minimal hypersurfaces $M^n\subset \mathbf{S}^{n+1}$ states that
\begin{equation}
|{A}|\Lap|{A}|+|{A}|^4 \geq \frac{2}{n}|\nabla |{A}||^2 + n|{A}|^2.
\end{equation}
The method is then to apply $|A|$ as a test function in the variational characterisation of the first stability eigenvalue $\lambda_1$; in the varifold setting the presence of the singular set presents a major issue. To overcome this obstacle, we use a careful choice of cutoff functions that allows us to establish effective $L^2$ estimates for $|A|$ on small balls in $\mathbf{S}^{n+1}$. We then adapt the Schoen-Simon-Yau \cite{SSY} technique to upgrade these estimates to an $L^4$ bound on $|A|$, which establishes that it is a valid test function. 

In \cite{zhu} we have also applied Theorem \ref{thm:stabilityeigsphintro} to show that there are no non-flat entropy-stable stationary cones in $\mathbf{R}^k$, for any dimension $k$. 

In \cite{MR} Morgan and Ritor\'{e} described an alternative cutoff approach that was used to analyse stability of constant mean curvature hypersurfaces and the isoperimetric problem. Their construction is more involved than the one we use but has the benefit of producing smooth cutoff functions with integral control on the Laplacian, although neither of these properties is necessary for our argument. For completeness we present their approach in the appendix. 

Let us now briefly outline the structure of this paper. First, in Section \ref{sec:prelims}, we detail our notation and conventions. In Section \ref{sec:intsing} we describe our choice of cutoff functions and use them to establish some results allowing us to integrate by parts around the singular set. Assuming that $\lambda_1$ is finite, we then prove our integral estimates for the second fundamental form in Section \ref{sec:intest}. Finally, in Section \ref{sec:stabeig} we complete the proof of Theorem \ref{thm:stabilityeigsphintro}.

For some intermediate results, the hypotheses on the singular set may be weakened using the regularity theory for stable minimal hypersurfaces (see Proposition \ref{prop:regularity}), but we will state those results with the stronger hypotheses in order to clarify the dependence on the size of the singular set.


\begin{acknowledgements}
The author would like to thank Prof. Frank Morgan for bringing the paper \cite{MR} to his attention. 

This work was supported in part by the National Science Foundation under grant DMS-1308244.
\end{acknowledgements}

\section{Preliminaries}
\label{sec:prelims}

\subsection{Notation and background}
\label{sec:notation}

\subsubsection{Hypersurfaces}

In this paper a hypersurface will always mean a $C^2$ embedded codimension 1 submanifold in a smooth Riemannian manifold. We mainly consider hypersurfaces $M^n\subset \mathbf{S}^{n+1}$, where $\mathbf{S}^{n+1}$ is the round unit sphere. 

We write $\nabla$ for the connection on $M$ and $\ol{\nabla}$ for the ambient connection. Our convention for the Laplacian is \begin{equation} \Lap_M f = \div_M(\nabla^M f).\end{equation} If $M$ is two-sided, there is a well-defined normal field $\nu$ and we denote by $A$ the second fundamental form of $M$ along $\nu$. We take the mean curvature on $M$ to be \begin{equation}H = \div_M \nu.\end{equation} We say that $M$ is minimal if its mean curvature is zero. Note that since the ambient space $\mathbf{S}^{n+1}$ is orientable, a hypersurface $M^n\subset \mathbf{S}^{n+1}$ is two-sided if and only if it is orientable (see for instance \cite[Chapter 4]{hirsch}). 

We say that a hypersurface $M^n\subset\mathbf{S}^{n+1}$ has Euclidean volume growth if there exists a constant $C_V>0$ so that $\vol(M\cap B_r(x)) \leq C_V r^n$ for any $r>0$ and any $x\in\mathbf{S}^{n+1}$. Here, and henceforth, $B_r(x)$ denotes the geodesic ball of radius $r$ in $\mathbf{S}^{n+1}$ centred at $x$. 

\subsubsection{Varifolds}

In this paper a varifold will always mean an integer rectifiable (integral) varifold. We will consider integral $n$-varifolds $V$ in $\mathbf{S}^{n+1}\subset \mathbf{R}^{n+2}$. We write $\mathcal{H}^k$ for the $k$-dimensional Hausdorff measure in $N$. The reader is directed to \cite{simon} for the basic definitions for varifolds. An integral varifold $V$ is determined by its mass measure, which we denote $\mu_V$. We will always assume that the support $\supp V = \supp\mu_V$ is connected. We define the regular part $\reg V$ to be the set of points $x \in \supp V$ around which $\supp V$ is locally a $C^2$ hypersurface. 
The singular set is then $\sing V = \supp V \setminus \reg V$. 

An integer rectifiable $n$-varifold $V$ has an approximate tangent plane $T_x V$ at $\mu_V$-almost every $x$ in $\supp V$. We may thus define the divergence almost everywhere by \begin{equation} (\div_V X) (x)= \div_{T_x V} X(x) = \sum_{i=1}^n \langle E_i,\ol{\nabla}_{E_i} X\rangle(x)\end{equation} where $E_i$ is an orthonormal basis for $T_x V$ and $\ol{\nabla}$ is the ambient connection. 

For convenience will say that a varifold $V$ is orientable if and only if $\reg V$ is orientable. 

For most of our results we will need some control on the singular set, although we will not assume any such control for now. The weakest condition we will use is the $\alpha$-structural hypothesis of Wickramasekera (\cite{wick}, see also \cite[Section 12]{CMgeneric}): An $n$-varifold $V$ satisfies the $\alpha$-structural hypothesis for some $\alpha\in(0,1)$, if no point of $\sing V$ has a neighbourhood in which $\supp V$ corresponds to the union of at least three embedded $C^{1,\alpha}$ hypersurfaces with boundary that meet only along their common $C^{1,\alpha}$ boundary. Note that the $\alpha$-structural hypothesis is automatically satisfied if, for instance, $\sing V$ has vanishing codimension 1 Hausdorff measure. 

We say that a $n$-varifold $V$ in $\mathbf{S}^{n+1}$ has Euclidean volume growth if there exists a constant $C_V>0$ so that $\mu_V(B_r(x))  \leq C_V r^n$ for any $r>0$ and any $x\in\mathbf{S}^{n+1}$. 

\subsubsection{Stationary varifolds}

An integral $n$-varifold $V$ in $\mathbf{S}^{n+1} \subset \mathbf{R}^{n+2}$ is said to be stationary if \begin{equation} \int \div_V X \d\mu_V = 0\end{equation} for any $C^1$ vector field $X$ on $\mathbf{S}^{n+1}$. In particular the regular part must be minimal in $\mathbf{S}^{n+1}$. 

It follows from the monotonicity formula (see \cite[Section 17.6]{simon}) that any stationary varifold $V$ in $\mathbf{S}^{n+1}$ has Euclidean volume growth. 

We will need the Simons' inequality \cite{simonsjim} for minimal hypersurfaces in $\mathbf{S}^{n+1}$, (see also \cite{SSY}): 
\begin{lemma}
On any minimal hypersurface $M^n\subset \mathbf{S}^{n+1}$ we have that 
\begin{equation}
\label{eq:simonseq}
\Lap |{A}|^2=2|\nabla {A}|^2+2n|{A}|^2-2|{A}|^4.
\end{equation}
Consequently, 
\begin{eqnarray}
\label{eq:simonsineqsph}
|{A}|\Lap|{A}| &=& |\nabla A|^2 - |\nabla |A||^2 + n|A|^2 -|A|^4 \\\nonumber&\geq&  \frac{2}{n}|\nabla |{A}||^2 + n|{A}|^2 - |A|^4.
\end{eqnarray}
\end{lemma}

\subsubsection{Connectedness}

It will be useful to record a connectedness lemma that follows from the varifold maximum principle of Wickramasekera \cite[Theorem 19.1]{wick} together with the proof of \cite[Theorem A(ii)]{ilmanen96max} (see also \cite[Lemma 1.2]{zhu}). 

\begin{lemma}
\label{lem:connectedness}
Let $V$ be a stationary integral $n$-varifold in $\mathbf{S}^{n+1}$. Suppose $\mathcal{H}^{n-1}(\sing V)=0$. Then $\reg V$ is connected if and only if $\supp V$ is connected. 
\end{lemma}

\subsubsection{Stability eigenvalues}

For hypersurfaces $M^{n}$ in $\mathbf{S}^{n+1}$, we will consider the usual stability operator for area given by (recalling that $\mathbf{S}^{n+1}$ has constant Ricci curvature $n$) \begin{equation} {L} = \Lap_M + |{A}|^2 + n. \end{equation} Our convention is that $u$ is an eigenfunction of $L$ with eigenvalue $\lambda$ if $Lu=-\lambda u$. 
For any domain $\Omega \ssubset  M$ we can consider the Dirichlet eigenvalues $\{\lambda_i(\Omega)\}_{i\geq 1}$, and we define the first stability eigenvalue of $M$ to be \begin{equation}\label{eq:eigdefn}\lambda_1(M) = \inf_{\Omega} \lambda_1(\Omega) = \inf_f \frac{\int_M \left(|\nabla^M f|^2-|{A}|^2f^2- n f^2\right)}{\int_M f^2}.\end{equation}
Here the infimum may be taken over Lipschitz functions $f$ with compact support in $M$, although it could be $-\infty$. 

If, however, $\lambda_1=\lambda_1(M)>-\infty$, then we immediately get the stability inequality 
\begin{equation}
\label{eq:stabilityineqsph}
\int_M |{A}|^2 \phi^2 \leq \int_M |\nabla \phi|^2 + (-\lambda_1-n)\int_M \phi^2
\end{equation}
for any $\phi$ with compact support in $M$. It follows easily that $\lambda_1 \leq -n$ with equality iff $M$ is totally geodesic. 

If $V$ is an orientable stationary integral $n$-varifold in $\mathbf{S}^{n+1}$, we set $\lambda_1(V)=\lambda_1(\reg V)$. 

\subsection{Regularity theory}


Here we record a regularity result for stationary varifolds $V$ in $\mathbf{S}^{n+1}$ with $\lambda_1(V)=\lambda_1(M)>-\infty$ that satisfy the $\alpha$-structural hypothesis, where $M=\reg V$. 

This following proposition follows from the regularity theory of Wickramasekera \cite{wick} for stable (that is, $\lambda_1\geq 0$) stationary varifolds, as presented in \cite[Section 12]{CMgeneric}. The key is that the regularity theory goes through even if one only assumes a slightly weaker stability inequality of the form $\int_M |A|^2 \phi^2 \leq (1+\epsilon)\int_M |\nabla\phi|^2$; see for instance \cite[Proposition 12.25]{CMgeneric}, which holds for small balls in any fixed Riemannian manifold. Similarly, a version of \cite[Lemma 12.7]{CMgeneric} holds on any closed ambient manifold, and it holds with any lower bound $\lambda_1(M)>-\infty$ because the $\int_M \phi^2$ term in (\ref{eq:stabilityineqsph}) can be controlled on small balls (by the Poincar\'{e} inequality) to be small relative to the $\int_M |\nabla \phi|^2$ term. 

\begin{proposition}
\label{prop:regularity}
Let $V$ be an orientable stationary $n$-varifold in $\mathbf{S}^{n+1}$, satisfying the $\alpha$-structural hypothesis for some $\alpha\in(0,\frac{1}{2})$. Suppose that $\lambda_1(V) >-\infty$. Then $V$ corresponds to an embedded, analytic hypersurface away from a closed set of singularities of Hausdorff dimension at most $n-7$ (that is empty if $n\leq 6$ and discrete if $n=7$.)
\end{proposition}

Alternatively, one may also recover Proposition \ref{prop:regularity} from the regularity theory for self-shrinkers in \cite{CMgeneric}, by noting that the cone over any stationary varifold in $\mathbf{S}^{n+1}$ is a stationary cone in $\mathbf{R}^{n+2}$, hence a self-shrinker, with finite entropy (see for instance \cite{zhu}).

\section{Integration on singular hypersurfaces}
\label{sec:intsing}

In this section we present some technical results that will allow us to work on the regular part of an integral varifold with sufficiently small singular set.

 \subsection{Cutoff functions}
 \label{sec:cutoff}

We first detail our choice of cutoff functions that will allow us to integrate around the singular set, so long as that set is small enough.

Consider an integral $n$-varifold $V$ in $\mathbf{S}^{n+1}$. The singular set is closed and hence compact. So by definition of Hausdorff measure, if $\mathcal{H}^{n-q}(\sing V)=0$, then for any $\epsilon>0$ we may cover the singular set by finitely many geodesic balls of $\mathbf{S}^{n+1}$, $\sing V \subset \bigcup_{i=1}^m {B}_{r_i}(p_i)$, where $\sum_i r_i^{n-q} <\epsilon$ and we may assume without loss of generality that $r_i<1$ for each $i$. 

Given such a covering we take smooth cutoff functions $0\leq \phi_{i,\epsilon}\leq 1$ on $\mathbf{S}^{n+1}$ with $\phi_{i,\epsilon}=1$ outside ${B}_{2r_i}(p_i)$, $\phi_{i,\epsilon}=0$ inside ${B}_{r_i}(p_i)$ and $|\ol{\nabla} \phi_{i,\epsilon}| \leq \frac{2}{r_i}$ in between. We then set $\phi_\epsilon = \inf_i \phi_{i,\epsilon}$, which is Lipschitz with compact support away from $\sing V$,  and $|\ol{\nabla} \phi_\epsilon| \leq \sup_i |\ol{\nabla} \phi_{i,\epsilon}|$. 

\subsection{Integration by parts}

In this section $M^n$ will denote the regular part of an integral $n$-varifold $V$ in $\mathbf{S}^{n+1}$ with Euclidean volume growth $\mu_V(B_r(x)) \leq C_V r^n$. The main goal is to establish conditions under which integration by parts is justified on $M$. Henceforth, $L^p = L^p(M)$ and $W^{k,p}=W^{k,p}(M)$ will denote the usual function spaces on $M$. 

\begin{lemma}
\label{lem:gradconvsph}
Suppose that $\mathcal{H}^{n-q}(\sing V)=0$ for some $q>0$. Let $\phi_\epsilon$ be as in Section \ref{sec:cutoff}. Then on $M=\reg V$ we have the following gradient estimate:
\begin{equation}\int_M |\nabla \phi_\epsilon|^q  \leq 2^{n+q}C_V\epsilon.\end{equation} 
\end{lemma}
\begin{proof}
We have \begin{eqnarray}
\int_M |\nabla \phi_\epsilon|^{q} &\leq&  \sum_{i=1}^N \int_{M \cap {B}_{2r_i}(p_i) \setminus {B}_{r_i}(p_i)} \frac{2^{q}}{r_i^{q}} \\\nonumber&\leq& 2^{n+q}C_V\sum_{i} r_i^{n-q} \leq 2^{n+q}C_V\epsilon.\end{eqnarray} 
\end{proof}

\begin{corollary}
\label{cor:fgradsph}
Assume that $\mathcal{H}^{n-q}(\sing V)=0$ for some $q$ and let $M=\reg V$. Let $\phi_\epsilon$ be as in Section \ref{sec:cutoff}.
\begin{enumerate}
\item Suppose that $q\geq 1$ and that $f$ is $L^p$ on $M$, where $p = \frac{q}{q-1}$. Then \begin{equation}\lim_{\epsilon\rightarrow 0} \int_M |f||\nabla \phi_\epsilon| = 0.\end{equation} 
\item Suppose that $q\geq 2$ and that $f$ is $L^p$ on $M$, where $p= \frac{2q}{q-2}$. Then \begin{equation}\lim_{\epsilon\rightarrow 0} \int_M f^2 |\nabla \phi_\epsilon|^2 = 0.\end{equation} 

\end{enumerate}
\end{corollary}
\begin{proof}
For (1), using H\"{o}lder's inequality, we have
\begin{equation}
\int_M |f| |\nabla \phi_\epsilon|  \leq \|f\|_p \left( \int_M |\nabla \phi_\epsilon|^{q}\right)^\frac{1}{q}
\end{equation}
where $\frac{1}{p}+\frac{1}{q}=1$. 

Similarly for (2) we have
\begin{equation}
\int_M f^2 |\nabla \phi_\epsilon|^2  \leq \|f\|^2_p\left( \int_M |\nabla \phi_\epsilon|^{q}\right)^\frac{2}{q}
\end{equation}
where $\frac{2}{p}+\frac{2}{q}=1$. 

By supposition the $L^p$-norms of $f$ are finite, so both results now follow from Lemma \ref{lem:gradconvsph}. 
\end{proof}


\begin{lemma}
\label{lem:ibpsph}
Suppose $\mathcal{H}^{n-q}(\sing V)=0$ for some $q\geq 1$. Further suppose that $u,v$ are $C^2$ functions on $M=\sing V$ such that $|\nabla u||\nabla v|$ and $|u\Lap v|$ are $L^1$, and $|u\nabla v|$ is $L^p$, $p= \frac{q}{q-1}$. Then \begin{equation} 
\int_M u\Lap v  = -\int_M \langle \nabla u,\nabla v\rangle .\end{equation}
\end{lemma}
\begin{proof}
If $\phi$ has compact support we may integrate by parts to get \begin{equation}
\int_M \phi u\Lap v = -\int_M\phi \langle \nabla u,\nabla v\rangle - \int_M   u\langle \nabla v,\nabla \phi\rangle. \end{equation} 

Applying this to $\phi= \phi_\epsilon$, Corollary \ref{cor:fgradsph} gives that the second term on the right tends to zero as $\epsilon\rightarrow 0$, so the result follows by dominated convergence.

\end{proof}

\section{Integral estimates for $|{A}|$}
\label{sec:intest}

Throughout this section $M^n$ will denote the regular part of an orientable stationary $n$-varifold $V$ in $\mathbf{S}^{n+1}$, which satisfies $\lambda_1=\lambda_1(V)>-\infty$. Recall that $V$ automatically has Euclidean volume growth, and that the stability inequality 
\begin{equation}
\label{eq:stabilitysph}
\int_M |{A}|^2 \phi^2 \leq \int_M |\nabla \phi|^2 + (-\lambda_1-n)\int_M \phi^2
\end{equation}
holds for any $\phi$ with compact support in $M$. The goals of this section are to provide $L^2$ estimates for the second fundamental form $A$ on small balls, and to use these estimates to show that $|A|$ is $L^4$ on $M$. 

\begin{lemma}
\label{lem:AL2sph}
Suppose that $\mathcal{H}^{n-2}(\sing V)=0$ and let $M=\reg V$. If $\lambda_1>-\infty$, then $|{A}|$ is $L^2$ on $M$.\end{lemma}
\begin{proof}
Let $\phi_\epsilon$ be as in Section \ref{sec:intsing}. Plugging $\phi_\epsilon$ into the stability inequality (\ref{eq:stabilitysph}), the last term is bounded by the volume of $M$ since $\phi_\epsilon^2\leq 1$. The gradient term is controlled by Lemma \ref{lem:gradconvsph}, so the result follows by Fatou's lemma as we take $\epsilon\rightarrow 0$.
\end{proof}

\begin{lemma}
\label{lem:AL2locsph}
Suppose that $\mathcal{H}^{n-2}(\sing V)=0$ and let $M=\reg V$. Further suppose that $\lambda_1>-\infty$. Then there exists $C=C(n,V,\lambda_1)$ so that for any $r\in(0,2)$, $p\in \mathbf{S}^{n+1}$, we have \begin{equation}
\int_{M \cap {B}_r(p)} |{A}|^2 \leq C r^{n-2}.\end{equation}
\end{lemma}
\begin{proof}
First we fix a cutoff function $0\leq \eta \leq 1$ so that $\eta=1$ inside ${B}_r(p)$, $\eta=0$ outside ${B}_{2r}(p)$ and $|\ol{\nabla} \eta | \leq \frac{2}{r}$ in between. Then \begin{equation} \int_{M \cap {B}_r(p)} |{A}|^2 \leq \int_{M} |{A}|^2\eta^2.\end{equation}

Since by Lemma \ref{lem:AL2sph}, $|{A}|$ is $L^2$, these integrals are finite and using dominated convergence we may approximate 
\begin{equation}
\int_M |{A}|^2 \eta^2  =\lim_{\epsilon\rightarrow 0} \int_{M} |{A}|^2 \eta^2 \phi_\epsilon^2.
\end{equation}

Now using the stability inequality (\ref{eq:stabilitysph}), for each $\epsilon>0$ we have
\begin{equation} 
\int_M |{A}|^2 \eta^2 \phi_\epsilon^2\leq 2\int_M \eta^2|\nabla \phi_\epsilon|^2 +2\int_M\phi_\epsilon^2|\nabla \eta|^2+\alpha \int_M \eta^2 \phi_\epsilon^2 ,
\end{equation}
where we have set $\alpha = |-\lambda_1-n|$. 

Since $\eta^2\leq 1$, by Lemma \ref{lem:gradconvsph} the first term tends to zero as $\epsilon\rightarrow 0$, where we have used that $\mathcal{H}^{n-2}(\sing V)=0$. Then since also $\phi_\epsilon^2\leq 1$, we have \begin{equation}\int_M \eta^2 \phi_\epsilon^2 \leq \int_M \eta^2 \leq \int_{M \cap {B}_{2r}(p)} 1 \leq 2^{n} C_V r^{n} ,\end{equation} and 
\begin{equation}
\int_M  \phi_\epsilon^2 |\nabla \eta|^2 \leq \int_M |\nabla \eta|^2 \leq \int_{M \cap {B}_{2r}(p)} \frac{4}{r^2}\leq 2^{n+2} C_V r^{n-2},
\end{equation}
so the result follows.
\end{proof}

\begin{lemma}
\label{lem:ssysph}
Suppose that $\mathcal{H}^{n-4}(\sing V)=0$ and let $M=\reg V$. If $\lambda_1>-\infty$, then $|{A}|$ is $L^4$, and $|\nabla|{A}||$ and $|\nabla {A}|^2$ are $L^2$, on $M$. 
\end{lemma}
\begin{proof}
We adapt the Schoen-Simon-Yau \cite{SSY} type argument. 

Set $\alpha= |-\lambda_1-n|$. If $f$ has compact support in $ M$ then applying the stability inequality (\ref{eq:stabilitysph}) with $\phi=|{A}|f$ and using the absorbing inequality yields
\begin{equation}
\label{eq:ssy1sph}
\int_M |{A}|^4 f^2 \leq (1+a)\int_M |\nabla |{A}||^2 f^2 + \int_M |{A}|^2\left((1+a^{-1})|\nabla f|^2 + \alpha f^2\right),
\end{equation}
where $a>0$ is an arbitrary positive number to be chosen later. 

On the other hand, multiplying the Simons' inequality (\ref{eq:simonsineqsph}) by $f^2$ (dropping the harmless $n|{A}|^2$ term), integrating by parts and again using the absorbing inequality gives \begin{equation}
\label{eq:ssy2sph}
\int_M |{A}|^4 f^2 + a^{-1} \int_M |{A}|^2|\nabla f|^2 \geq \left(1+\frac{2}{n}-a\right) \int_M |\nabla |{A}||^2f^2. 
\end{equation}

Combining these gives
\begin{equation}
\int_M |{A}|^4 f^2 \leq \frac{1+a}{1+\frac{2}{n}-a} \int_M |{A}|^4 f^2 + C_a \int_M |{A}|^2(|\nabla f|^2+\alpha f^2).
\end{equation} 
Choosing $a<\frac{1}{n}$ gives that the first coefficient on the right is less than 1 and hence may be absorbed on the left, thus
\begin{equation}
\int_M |{A}|^4 f^2 \leq  C\int_M |{A}|^2 (f^2+|\nabla f|^2),
\end{equation}
where $C=C(n,\alpha)$. 

We now apply this inequality with $f = \phi_\epsilon$, where $\phi_\epsilon$ is as in Section \ref{sec:intsing} with $q=4$. As $\epsilon\rightarrow 0$, the first term on the right converges to $\int_M |{A}|^2$, which we already know is finite. We bound the second term as follows:
\begin{equation}
\int_M |{A}|^2 |\nabla \phi_\epsilon|^2 \leq \sum_{i=1}^m \frac{4}{r_i^2}\int_{M \cap {B}_{2r_i}(p_i)\setminus {B}_{r_i}(p_i)} |{A}|^2.
\end{equation}
Using Lemma \ref{lem:AL2locsph}, we have that
\begin{eqnarray}
\int_{M\cap {B}_{2r_i(p_i)}\setminus {B}_{r_i}(p_i)} |{A}|^2 \leq C' r_i^{n-2},
\end{eqnarray}
where $C'$ depends on $\alpha$ and the volume bounds for $M$. 

Therefore
\begin{equation}
\label{eq:ssyfin}
\int_M |{A}|^2|\nabla \phi_\epsilon|^2  \leq 4C'\sum_i r_i^{n-4} < 4C'\epsilon.
\end{equation}
where we recall that the $r_i$ were chosen so that $\sum_i r_i^{n-4} < \epsilon$. 

Taking $\epsilon\rightarrow 0$ we see that this term tends to 0, thus we have shown that indeed $|{A}|$ is $L^4$. With this fact in hand, it follows from (\ref{eq:ssy2sph}) that $|\nabla |{A}||$ is $L^2$.

Finally, multiplying the identity (\ref{eq:simonseq}) by $f^2$ and integrating by parts, we have that 
\begin{equation}
\int_M f^2(|\nabla {A}|^2-|{A}|^4)\rho \leq -\int_M 2f|{A}|\langle \nabla f,\nabla |{A}|\rangle  \leq \int_M (f^2|\nabla|{A}||^2 + |{A}|^2|\nabla f|^2). 
\end{equation}
Since we now know that $|\nabla |{A}||$ is $L^2$ and that $|{A}|$ is $L^4$, we again set $f=\phi_\epsilon$ and use (\ref{eq:ssyfin}) to control the last term; this shows that $|\nabla {A}|^2$ is $L^2$, as desired.

\end{proof}

\section{First stability eigenvalue}
\label{sec:stabeig}

In this section we prove our main theorem that the first stability eigenvalue of a stationary $n$-varifold in $\mathbf{S}^{n+1}$ is at most $-2n$. First we need the following lemma:

\begin{lemma}
 \label{lem:lipcutsph}
 Let $V$ be a stationary integral $n$-varifold in $\mathbf{S}^{n+1}$, with orientable regular part. If $\mathcal{H}^{n-4}(\sing V)=0$ then we get the same $\lambda_1=\lambda_1(V)$ by taking the infimum over Lipschitz functions $f$ on $ M=\reg V$ such that $f\in W^{1,2} \cap L^4$. \end{lemma}
\begin{proof}
Obviously we may assume that $\lambda_1>-\infty$. Then by Lemma \ref{lem:AL2sph} and since $f$ is $L^4$, we also have that $|A|f$ is also $L^2$. 
We will use the functions $f_\epsilon = f \phi_\epsilon$, which are compactly supported away from the singular set, in the definition (\ref{eq:eigdefn}) of $\lambda_1$. 

Now since $f$ and $|A|f$ are $L^2$, dominated convergence gives that $\int_M f_\epsilon^2 \rightarrow \int_M f^2$ and $\int_M |A|^2 f_\epsilon^2 \rightarrow \int_M |A|^2 f^2$ as $\epsilon\rightarrow 0$. For the gradient term we have
\begin{equation}
\label{eq:lipcutsph}
\int_M |\nabla f_\epsilon|^2 =  \int_M (\phi_\epsilon^2|\nabla f|^2 + 2\langle \nabla f,\nabla\phi_\epsilon\rangle + f^2 |\nabla \phi_\epsilon|^2  ).
\end{equation}
Since $f$ is $L^4$, parts (1) and (2) respectively of Corollary \ref{cor:fgradsph} give that the second and third terms on the right tend to zero as $\epsilon\rightarrow 0$. Moreover, the first term on the right tends to $\int_M |\nabla f|^2$ by dominated convergence. Thus we have shown that $\int_M |\nabla f_\epsilon|^2  \rightarrow \int_M |\nabla f|^2$, and the lemma follows. 
\end{proof}

We now proceed to the proof of Theorem \ref{thm:stabilityeigsphintro}.

\begin{theorem}
\label{thm:stabilityeigsph}
Let $V$ be a stationary integral $n$-varifold in $\mathbf{S}^{n+1}$, with orientable regular part. Suppose that $V$ satisfies the $\alpha$-structural hypothesis for some $\alpha\in(0,\frac{1}{2})$.

Further suppose that $V$ is not totally geodesic in $\mathbf{S}^{n+1}$. Then $\lambda_1(V)\leq -2n$, with equality if and only if $V$ is an integer multiple of a Clifford hypersurface $ \mathbf{S}^k\left(\sqrt{\frac{k}{n}}\right)\times \mathbf{S}^l\left(\sqrt{\frac{l}{n}}\right)$, where $k+l=n$. 
\end{theorem}
\begin{proof}
Set $M=\reg V$. 

Obviously we may assume $\lambda_1 >-\infty$. The regularity theory Proposition \ref{prop:regularity} then implies that $\sing V$ has codimension at least 7. Thus we certainly have that $\mathcal{H}^{n-4}(\sing V)=0$. 

So by Lemma \ref{lem:ssysph}, we have that $|{A}|$ is $L^4$ (and $L^2$), and that $|\nabla |{A}||$ and $|\nabla {A}|$ are $L^2$. Therefore, by Lemma \ref{lem:lipcutsph}, we may use $|{A}|$ as a test function in the definition of $\lambda_1$, that is we have \begin{equation}\lambda_1(M) \leq \frac{\int_M (|\nabla |{A}||^2 -|{A}|^4 - n|{A}|^2)}{\int_M |{A}|^2}.\end{equation} Now we wish to integrate by parts. By (\ref{eq:simonsineqsph}) and our integral estimates for $A$ we have that $|A|\Lap|A|$ is $L^1$. Using Young's inequality we have \begin{equation}(|{A}|\, |\nabla |{A}||)^p\leq  \frac{2-p}{2} |{A}|^{\frac{2p}{2-p}} + \frac{p}{2}|\nabla |{A}||^2.\end{equation} Again since $|{A}|$ is $L^4$ and $|\nabla |A||$ is $L^2$, this implies that $|{A}|\, |\nabla |{A}||$ is $L^p$ for $p=\frac{4}{3}$. Thus now we may use Lemma \ref{lem:ibpsph} together with the Simons' inequality (\ref{eq:simonsineqsph}) to find that
\begin{eqnarray}
\label{eq:stabilityeigsph}
\int_M |\nabla |{A}||^2 &=& -\int_M |{A}|\Lap |{A}| \leq \int_M \left(-\frac{2}{n}|\nabla|{A}||^2+|{A}|^4 - n|{A}|^2\right) \\&\leq &\nonumber \int_M (|{A}|^4 - n|{A}|^2),
\end{eqnarray}
which implies that $\lambda_1(M)\leq -2n$ as claimed.

If $\lambda_1(M)=-2n$, then equality must hold in all previous inequalities. In particular for equality to hold in the last step of (\ref{eq:stabilityeigsph}) we must have $\int_M |\nabla |{A}||^2=0$ and $\int_M |A|^4 = n\int_M|A|^2$. Therefore $|{A}|$ is equal to a constant on $M$, and the constant must be $\sqrt{n}$. The Gauss equation and a theorem of Lawson \cite[Theorem 1]{lawson} (which does not assume completeness) then imply that $ M$ is a piece of a Clifford hypersurface $M_0= \mathbf{S}^k\left(\sqrt{\frac{k}{n}}\right)\times \mathbf{S}^l\left(\sqrt{\frac{l}{n}}\right)$, where $k+l=n$. The support $\supp V$ is then contained in $M_0$, so the constancy theorem implies that $V=m[M_0]$ for some integer $m$, as desired.
\end{proof}

\appendix

\section{A smooth cutoff construction}

In this appendix we reproduce the cutoff construction of Morgan and Ritor\'{e} \cite{MR}. We do so only for the case of submanifolds in Euclidean space with compact support and singular set of vanishing codimension 2 measure. The other cases considered in \cite{MR} are similar; see also Remark \ref{rmk:morgan}. In this section only, $B(x,r)$ denotes a Euclidean ball of radius $r$ centred at $x$. 

First we need an easy bound for the number of intersections of balls of comparable radii.   

\begin{lemma}
\label{lem:intersectionbd}
Let $\mathcal{B}=\{B(p_i,r_i)\}$ be a collection of balls in $\mathbb{R}^N$. Suppose that there are $\alpha,\beta\geq 1$ such that the sub-balls $\{B(p_i,r_i/\alpha)\}$ are pairwise disjoint, and that the radii are $\beta$-comparable, that is, $\sup r_i \leq \beta \inf r_i$. Then each ball in $\mathcal{B}$ intersects at most $(3\alpha\beta)^N-1$ other balls in $\mathcal{B}$.
\end{lemma}
\begin{proof}
Fix a ball $B(p_i,r_i) \in \mathcal{B}$. Any ball in $\mathcal{B}$ that intersects $B(p_i,r_i)$ must be contained in the larger ball $B(p_i, r_i + 2\sup r_j)$, which has volume at most $\omega_N (3\sup r_j)^N$. Here $\omega_N$ is the volume of the unit ball in $\mathbb{R}^N$. 

On the other hand, since the sub-balls $\{B(p_j,r_j/\alpha)\}$ are pairwise disjoint, each must take up a volume at least $\omega_N (\inf r_j/\alpha)^N$ in the larger ball $B(p_i, r_i + 2\sup r_j)$. So at most $\frac{(3\sup r_j)^N}{(\inf r_j/\alpha)^N} \leq (3\alpha\beta)^N$ sub-balls can fit in the larger ball, and this implies the result. 
\end{proof}

We now proceed to the construction of smooth cutoff functions around the singular set. 

\begin{proposition}[\cite{MR}]
Let $\Sigma^k$ be a smooth embedded submanifold in $\mathbb{R}^N$ with bounded mean curvature and compact closure $\ol{\Sigma}$. If $\sing\Sigma = \ol{\Sigma}\setminus\Sigma$ satisfies $\mathcal{H}^{k-2}(\sing \Sigma)=0$, then for any $\epsilon>0$ there exists a smooth function $\varphi_\epsilon: \ol{\Sigma} \rightarrow [0,1]$ supported in $\Sigma$ such that: 
\begin{enumerate}
\item $\mathcal{H}^k(\{\varphi_\epsilon\neq 1\}) <\epsilon$;
\item $\int_\Sigma |\nabla \varphi_\epsilon|^2 < \epsilon$;
\item $\int_\Sigma |\Lap \varphi_\epsilon| < \epsilon$.
\end{enumerate}
\end{proposition}
\begin{proof}
Fix a smooth radial cutoff function $\varphi:\mathbb{R}^N\rightarrow [0,1]$ such that $\phi=0$ in $B(0,1/2)$ and $\phi=1$ outside $B(0,1)$. The derivatives are bounded, say $|D\varphi|^2 + |D^2 \varphi| \leq C_0$. By scaling $\varphi$ to $B(x,r)$, $r\leq 1$, we get a cutoff function satisfying $|D\varphi|^2 + |D^2 \varphi| \leq C_0 r^{-2}$. 

Since $\Sigma$ has bounded mean curvature $|\vec{H}|\leq C_H$, the monotonicity formula \cite{simon} implies that there is a constant $C_V$ such that $\mathcal{H}^k(\Sigma \cap B(x,r)) \leq C_V r^k$ for any $r\leq 1$ and any $x$. Moreover, on $\Sigma \cap B_r(x)$ we will have \begin{equation} |\Lap \varphi| \leq k|D^2\varphi| + |\langle \vec{H}, D\varphi\rangle| \leq kC_0 r^{-2} + C_H \sqrt{C_0} r^{-1} \leq C_1 r^{-2}.\end{equation} 

Now let $\epsilon>0$. By definition of Hausdorff measure we may cover the singular set by finitely many balls $\{B(p_i,r_i/6)\}$ such that $r_i \leq 1$ and $\sum_{i} r_i^{k-2} <\epsilon$. Consider the enlarged cover $\{B(p_i,r_i/2)\}$. If $B(p_i,r_i/6)\cap B(p_j,r_j/6)\neq \emptyset$, $r_i \geq r_j$, then certainly $B(p_j, r_j/6) \subset B(p_i,r_i/2)$, so we could discard any such $j$. In doing so we obtain a cover $\{B(p_i,r_i/2)\}$ such that the $\{B(p_i,r_i/6)\}$ are pairwise disjoint. We may relabel the radii so that $r_1\geq r_2\geq \cdots$, and we partition the balls into classes of comparable radii $\mathcal{B}_m = \{i | 2^m \leq r_i < 2^{m+1}\}$. 

Cut off on each $B(p_i,r_i)$ by scaled cutoff functions $\varphi_i$ as above, then set $\varphi_\epsilon = \prod_i \varphi_i$. Immediately we have $\mathcal{H}^k(\{\varphi_\epsilon\neq 1\}) \leq \sum_i C_V r_i^{n} < C_V \epsilon$. For properties (2) and (3) we must bound the sum of all product terms $\int_\Sigma |\nabla \varphi_i||\nabla \varphi_j|$. Such a term is zero if $B(p_i,r_i)$ and $B(p_j,r_j)$ are disjoint, otherwise we have the bound \begin{equation} \label{eq:crossbd} \int_\Sigma |\nabla \varphi_i| |\nabla\varphi_j| \leq \frac{C_0 C_V}{r_i r_j} \min(r_i,r_j)^k.\end{equation} 

The procedure to estimate these cross terms is as follows: We fix $j$ and consider the sum over $i\leq j$ (that is, $r_i \geq r_j$). Letting $\mathcal{B}_{m_j}$ be the class containing $j$, we will bound the number of intersections that $B(p_j,r_j)$ can have with balls $B(p_i,r_i)$ in each class $\mathcal{B}_{m_j +h}$, $h\geq 0$. 

The key observation is that if $B(p_i,r_i) \cap B(p_j,r_j) \neq \emptyset$, then certainly the enlarged ball $B(p_i, r_i+r_j)$ must contain the point $p_j$. In particular all such enlarged balls must intersect each other. But for $i\in \mathcal{B}_{m_j+h}$, the radii $r_i+r_j$ are comparable to within a factor of \begin{equation}\frac{2^{m_j+h+1} + r_j}{2^{m_j+h}+r_j} \leq \frac{2^{m_j+h+1} + 2^{m_j+1}}{2^{m_j+h}+2^{m_j}} = 2.\end{equation} For any such $i$ we also have $\frac{r_j}{r_i} \leq \frac{2^{m_j+1}}{2^{m_j+h}} = 2^{1-h},$ so in particular $\frac{r_i+r_j}{r_i} \leq 3$ and the sub-balls $\{B(p_i,\frac{r_i+r_j}{18}) | i \in \mathcal{B}_{m_j+h}\}$ must be pairwise disjoint. Thus by Lemma \ref{lem:intersectionbd}, \begin{equation}\label{eq:mr3}\#\{i \in \mathcal{B}_{m_j+h} | B(p_i,r_i) \cap B(p_j,r_j) \neq \emptyset \} \leq \#\{ i \in \mathcal{B}_{m_j+h} | B(p_i, r_i+r_j) \ni p_j\} \leq 108^N.\end{equation} 

Using again that $\frac{r_j}{r_i} \leq 2^{1-h}$ for $i\in\mathcal{B}_{m_j+h}$, the estimate (\ref{eq:crossbd}) gives \begin{equation}\int_\Sigma |\nabla \varphi_i| |\nabla\varphi_j| \leq C_0 C_V r_j^{k-1} r_i^{-1} \leq C_0 C_V 2^{1-h} r_j^{k-2}.\end{equation} Then by (\ref{eq:mr3}) we have that\begin{equation}\sum_{\substack{i\leq j\\ i \in\mathcal{B}_{m_j+h}}} \int_\Sigma |\nabla \varphi_i| |\nabla\varphi_j| \leq 108^N C_0 C_V 2^{1-h} r_j^{k-2},\end{equation} and summing over $h\geq 0$ we get that $\sum_{i\leq j} \int_\Sigma |\nabla \varphi_i| |\nabla\varphi_j| \leq 4(108^N)C_0 C_V r_j^{k-2}.$

Finally, summing over $j$ gives  \begin{eqnarray}\sum_{i,j} \int_\Sigma |\nabla \varphi_i| |\nabla\varphi_j| &=& 2\sum_j \sum_{i\leq j} \int_\Sigma |\nabla \varphi_i| |\nabla\varphi_j| \\ \nonumber &\leq& 8(108^N)C_0 C_V \sum_j r_j^{k-2} < 8(108^N)C_0 C_V\epsilon. \end{eqnarray}

Since each $\varphi_j^2\leq1$ we conclude that \begin{equation} \int_\Sigma |\nabla \varphi_\epsilon|^2 \leq \sum_{i,j} \int_\Sigma |\nabla \varphi_i| |\nabla\varphi_j| <8(108^N)C_0 C_V\epsilon\end{equation} and \begin{eqnarray} \int_\Sigma |\Lap \varphi_\epsilon| &\leq& \sum_i \int_\Sigma |\Lap \varphi_i| + \sum_{i,j} \int_\Sigma |\nabla \varphi_i| |\nabla\varphi_j|  \\\nonumber&<&  \sum_i C_1C_V r_i^{k-2} +8(108^N)C_0 C_V\epsilon < (C_1+8(108^N)C_0)C_V\epsilon .\end{eqnarray}

Since $\epsilon$ was arbitrary this concludes the proof. 
\end{proof}

\begin{remark}
\label{rmk:morgan}
Morgan and Ritor\'{e} \cite{MR} state their construction also for somewhat more general assumptions as follows. If the ambient space is a regular cone, the argument proceeds with one extra cutoff around the vertex. If $k=2$ and the surface has isolated singular points, one may use logarithmic cutoff functions (and one does not need to be concerned with the intersections). Finally, to handle the noncompact case one may proceed by covering the singular set in annuli $B_{m+1}\setminus B_{m-1}$ with balls of radius $r_{m,i}< 1$ such that $\sum_i r_{m,i}^{k-2} \leq \frac{2^{-n}\epsilon}{C_V(m)}$. \end{remark}

\bibliographystyle{plain}
\bibliography{sphstabilityv9}
\end{document}